\documentclass[11pt]{article}

\pdftrailerid{}
\usepackage[T1]{fontenc}
\usepackage[utf8]{inputenc}
\usepackage{lmodern,amsmath,amssymb,amsthm,mathtools,mathrsfs}
\usepackage[a4paper,margin=1.02in]{geometry}
\usepackage{booktabs,longtable,array,microtype}
\usepackage{xcolor}
\usepackage[colorlinks=true,linkcolor=blue!50!black,
  citecolor=blue!50!black,urlcolor=blue!50!black]{hyperref}
\usepackage{enumitem,needspace}
\setlist[itemize]{leftmargin=1.5em,itemsep=2pt,topsep=4pt}
\setlist[enumerate]{leftmargin=1.8em,itemsep=2pt,topsep=4pt}
\allowdisplaybreaks
\numberwithin{equation}{section}
\newtheorem{theorem}{Theorem}[section]
\newtheorem{proposition}[theorem]{Proposition}
\newtheorem{lemma}[theorem]{Lemma}
\newtheorem{corollary}[theorem]{Corollary}
\theoremstyle{definition}

\theoremstyle{remark}
\newtheorem{remark}[theorem]{Remark}
\newcommand{\R}{\mathbb R}

\newcommand{\Z}{\mathbb Z}
\newcommand{\N}{\mathbb N}
\newcommand{\dd}{\,\mathrm d}
\newcommand{\Res}{\operatorname*{Res}}
\newcommand{\AC}{\operatorname{AC}}
\newcommand{\sgn}{\operatorname{sgn}}
\newcommand{\cP}{\mathcal P}
\newcommand{\cA}{\mathcal A}

\newcommand{\sbinom}[2]{\begin{bmatrix}#1\\#2\end{bmatrix}_{h}}
\hypersetup{
  pdftitle={Shifted poles and chamber cancellation for classical Witten zeta functions},
  pdfauthor={Jonas Matuzas},
  pdfsubject={Shifted poles on the positive real axis, chamber cancellation, and Laurent coefficients of classical Witten zeta functions},
  pdfkeywords={Witten zeta function, shifted pole, root system, Selberg integral,
    chamber cancellation, Laurent coefficient}}
\title{Shifted poles and chamber cancellation\\for classical Witten zeta functions}
\author{Jonas Matuzas\\
\href{mailto:jonas.matuzas@gmail.com}{\nolinkurl{jonas.matuzas@gmail.com}}}
\date{}
\begin{document}
\maketitle
\begin{abstract}
We determine two infinite families of poles on the positive real axis for classical single-variable Witten zeta functions.  In type $A_r$, for $r\geq5$, the point
\[
 q_r^A=\frac{2(r-4)}{r^2+r-4}
\]
is a simple pole except in ranks $12$ and $20$, where it is a double pole.  In type $D_r$, for $r\geq4$, the point
\[
 q_r^D=\frac{r-3}{r(r-1)-1}
\]
is a simple pole except at $D_8$, where it is double.  In root-product normalization, we express the simple residues and the leading coefficients of these three double poles in terms of gamma, trigonometric, and Riemann zeta values.

For $r\geq4$, the functions of types $B_r$ and $C_r$ are holomorphic at $q_r^{BC}=(r-3)/(r^2-1)$ except possibly in ranks $7$ and $11$, where any pole is simple. These pole and holomorphy statements arise from quadratic normal Taylor coefficients. For every fixed higher even normal degree, we also determine exactly when the associated continued $B/C$ chamber sum is nonzero; this auxiliary result does not by itself classify poles of the full Witten function. The proofs combine exhaustive support classifications, explicit integration-by-parts identities, and finite chamber relations. No numerical nonvanishing estimate is used.
\end{abstract}

\noindent\textbf{Generative-AI disclosure and author responsibility.}
OpenAI's ChatGPT 5.6 Pro was used for mathematical exploration, proof auditing,
drafting, and preparation of verification programs. The author is responsible
for the final manuscript.



\section{Introduction}
A candidate pole of a multivariable root-system zeta function may disappear when all exponents are set equal. We determine when this happens for three families of quadratic terms along parabolic faces. In types $A$ and $D$ the candidate values are poles in every rank under consideration. In types $B$ and $C$ most of these candidates are removable.

Let $\Phi$ be an irreducible crystallographic root system of rank $r$, and let $\Psi^+=\Phi^{\vee,+}$ be its set of positive coroots. Put $N=|\Psi^+|$ and write $\beta=\sum_i b_i(\beta)\beta_i$ in the simple-coroot basis. Throughout, $\N=\{1,2,\ldots\}$ and $\zeta$ denotes the Riemann zeta function.  We use the root-product normalization
\begin{equation}\label{eq:def-xi}
 \xi_\Phi(s)=\sum_{m_1,\ldots,m_r\geq1}
 \prod_{\beta\in\Psi^+}\left(\sum_{i=1}^r b_i(\beta)m_i\right)^{-s}.
\end{equation}
For type $A_r$ this is
\begin{equation}\label{eq:xi-A}
 \xi_{A_r}(s)=\sum_{m_1,\ldots,m_r\geq1}
 \prod_{1\leq i\leq j\leq r}(m_i+\cdots+m_j)^{-s}.
\end{equation}
The series is initially defined in its half-plane of absolute convergence; elsewhere $\xi_\Phi$ denotes its meromorphic continuation. The ordinary representation zeta function differs from \eqref{eq:def-xi} by a positive constant raised to the power $s$, so pole locations and orders are unchanged.

The three quadratic candidate values studied here are
\begin{equation}\label{eq:three-q}
 q_r^A=\frac{2(r-4)}{r^2+r-4},\qquad
 q_r^D=\frac{r-3}{r(r-1)-1},\qquad
 q_r^{BC}=\frac{r-3}{r^2-1}.
\end{equation}
The following theorems determine which of these candidates are poles.

\begin{theorem}[Type $A$]\label{thm:main}
For every $r\geq5$, $\xi_{A_r}$ has a pole at $q_r^A$.  The pole is simple for $r\notin\{12,20\}$ and double for $r\in\{12,20\}$.  In the simple cases its residue is the explicit formula \eqref{eq:A-residue}.  At the two exceptional ranks,
\begin{align}
 [(s-2/19)^{-2}]\xi_{A_{12}}(s)
 &=\frac{\Res_{s=2/19}\xi_{A_6}(s)}{228}\,
       \cP_{12,6}^{\mathrm{sum}}(2/19)>0,\label{eq:A12-main}\\
 [(s-1/13)^{-2}]\xi_{A_{20}}(s)
 &=\frac{\Res_{s=1/13}\xi_{A_5}(s)}{3120}\,
       \cP_{20,5}^{\mathrm{sum}}(1/13)<0.\label{eq:A20-main}
\end{align}
Here $[(s-q)^{-2}]f$ denotes the coefficient of $(s-q)^{-2}$ in the Laurent expansion of $f$. The positive face-period sums $\cP_{R,L}^{\mathrm{sum}}$ are defined in \eqref{eq:outer-period} and evaluated in Theorem~\ref{thm:outer-sum}.
\end{theorem}

\begin{theorem}[Type $D$]\label{thm:D-family}
For every $r\geq4$, $\xi_{D_r}$ has a pole at $q_r^D$. It is simple for $r\ne8$ and double for $r=8$. The simple residue is given by \eqref{eq:D-singleton-residue}; it is positive for $r=4$ and $r\geq9$, and negative for $5\leq r\leq7$. Moreover,
\begin{equation}\label{eq:D8-main}
 [(s-1/11)^{-2}]\xi_{D_8}(s)
 =\frac{\Res_{s=1/11}\xi_{D_4}(s)}{10560}
 S_3\!\left(\frac3{22},\frac{10}{11},-\frac1{22}\right)>0.
\end{equation}
Here $S_3$ is the Selberg gamma product \eqref{eq:Selberg}.
\end{theorem}

\begin{theorem}[Types $B$ and $C$]\label{thm:BC-family}
For every $r\geq4$ with $r\notin\{7,11\}$, both $\xi_{B_r}$ and $\xi_{C_r}$ are holomorphic at $q_r^{BC}$. At ranks $7$ and $11$ their pole order is at most one. The theorem does not decide whether $\xi_{B_7},\xi_{C_7},\xi_{B_{11}}$, or $\xi_{C_{11}}$ is holomorphic at the indicated point.
\end{theorem}

\paragraph{Main steps of the proof.}
There are two points at which a list of candidate hyperplanes is insufficient. First, the residue formula is expressed in the original simple-root gaps, whereas the chamber integrals are most easily evaluated after centering the collapsing blocks. Lemma~\ref{lem:centering} proves that this change of coordinates preserves the relevant projective period, including the mixed quadratic term. Proposition~\ref{prop:common-insertion} then removes the quadratic insertion from every type-$A$ chamber by an explicit integration-by-parts identity.

Second, a finite chamber coefficient may vanish at the same parameter at which its reference Selberg integral has a pole. The product must be evaluated from both Laurent expansions. Section~\ref{sec:cancellation} determines these orders, and Section~\ref{sec:actual-singletons} then adds every support that contributes to the full Witten coefficient.

\paragraph{Organization.}
Section~\ref{sec:local} states the residue formulas and fixes the integral conventions. Sections~\ref{sec:A-count}--\ref{sec:A-evaluation} prove Theorem~\ref{thm:main}, including the double-pole coefficients in Section~\ref{sec:cumulative}. Sections~\ref{sec:cancellation}--\ref{sec:actual-singletons} prove Theorems~\ref{thm:D-family} and~\ref{thm:BC-family}.

\subsection{Relation to earlier work}
Meromorphic continuation and the root-system framework are due to Matsumoto--Tsumura and Komori--Matsumoto--Tsumura \cite{MT,KMT}; related polynomial-Dirichlet continuation results include \cite{Lichtin,Essouabri}.

Au proved that $\xi_{A_3}$ is holomorphic at $s=1/5$ and that $\xi_{B_3},\xi_{C_3}$ are holomorphic at $s=1/8$ \cite[Theorems~9.2, 9.4, and 9.6]{AuRank}. In Remark~9.3, he anticipated that many other candidates obtained by diagonal restriction would not be poles. These observations motivate the cancellation problem considered here.

The two examples just cited have a singleton normal complement and Taylor degree $\ell=1$, in the notation of Section~\ref{sec:local}. Remark~\ref{rem:first-normal} gives the elementary first-order symmetry argument within this paper. Theorem~\ref{thm:BC-family} addresses Au's cancellation question for the quadratic singleton family, $\ell=2$. The results for higher even degrees in Section~\ref{sec:cancellation} classify the associated chamber sums, not the poles of the full Witten functions.

The local residue formulas are taken from \cite[Theorems 2.7, 3.2, 5.3, and 5.4]{Generic}. They apply to proper supports and to nested chains of supports. The comparison between their coefficients and the centered coordinates used here is a separate step, supplied by Lemma~\ref{lem:centering}.

The Selberg evaluation, mixed-chamber recurrences, and terminating basic-hypergeometric identities are classical \cite{FW,DLMF,Mimachi}. Here they are applied to the quadratic coefficients specified in \eqref{eq:three-q}. The new assertions are the two all-rank pole families, their three exceptional double-pole coefficients, and the $B/C$ quadratic holomorphy and endpoint criteria. The leading- and second-pole results \cite{Leading,Second,Walls} are earlier results used for comparison.

\section{Local residue formula and notation}\label{sec:local}
Let $I=\{1,\ldots,r\}$. A nonempty set $S\subseteq I$ specifies the simple-root coordinates that are scaled together. Define
\[
 R(S)=\{\beta\in\Psi^+: b_i(\beta)>0\text{ for some }i\in S\},
 \qquad N(S)=|R(S)|.
\]
The complementary nodes $J=I\setminus S$ generate the parabolic subsystem $\Phi_J$. We call it the normal complement: its coordinates are the normal variables in the Taylor expansion below. Thus
\begin{equation}\label{eq:count}
 N(S)=N-|\Psi_J^+|,
 \qquad q(S,\ell)=\frac{|S|-\ell}{N(S)}.
\end{equation}
The integer $\ell\geq0$ is the total Taylor degree in the normal variables. A candidate is \emph{shifted} when $\ell>0$. Its coefficient may vanish, so a candidate value need not be a pole. Positive candidates satisfy $0\leq\ell<|S|$. A \emph{singleton complement} means $|J|=1$, not $|S|=1$. In type $A$, two disjoint pairs correspond to two nonadjacent nodes of $J$.

Write $K_\Phi=\prod_{\beta\in\Psi^+}\operatorname{ht}(\beta)$, where $\operatorname{ht}(\beta)=\sum_i b_i(\beta)$ is the coroot height. The ordinary Witten normalization is $\zeta_\Phi(s)=K_\Phi^s\xi_\Phi(s)$. At a double pole $q$, coefficients $a_{-2},a_{-1}$ in root-product normalization become
\begin{equation}\label{eq:normalization}
 K_\Phi^q a_{-2},\qquad K_\Phi^q(a_{-1}+a_{-2}\log K_\Phi).
\end{equation}
In particular, a normalization change affects the simple coefficient beneath a double pole.

\Needspace{8\baselineskip}
\subsection{The analytic input}
We use four results of \cite{Generic}, specifically arXiv:2607.18945v2, in the normalization \eqref{eq:def-xi}. Theorem~2.7 gives the local flag expansion, Theorem~3.2 gives the residue on a proper support, Theorem~5.3 gives the recursive flag residue, and Theorem~5.4 gives the aggregate Laurent coefficients. Near a positive candidate $q$, the resulting expansion has the form below. Poles of lower-dimensional boundary integrals have already been assigned to longer support chains:
\begin{equation}\label{eq:local-form}
 \xi_\Phi(s)=H(s)+\sum_{\mathfrak f}
 \frac{C_{\mathfrak f}(s)}{(s-q)^{|\mathfrak f|}\prod_{S\in\mathfrak f}N(S)}.
\end{equation}
Here $H$ and each $C_{\mathfrak f}$ are holomorphic near $q$. A \emph{flag} $\mathfrak f$ is a strict chain $S_1\subsetneq\cdots\subsetneq S_k$ occurring in one resolved chart; its length is $|\mathfrak f|=k$. The sum includes the flags whose Taylor degrees satisfy $N(S_j)q-|S_j|+\ell_j=0$ at every step. An individual numerator may depend on the partition used in the resolution, but the sum at each Laurent order is independent of that choice.

The denominators in \eqref{eq:local-form} arise as follows. A resolved chart has factors
\[
 \prod_{j=1}^k\rho_j^{N(S_j)s-|S_j|-1}F(s,\rho,y),
 \qquad S_1\subsetneq\cdots\subsetneq S_k.
\]
Integrating a Taylor monomial of degree $\ell_j$ in $\rho_j$ gives the denominator $N(S_j)s-|S_j|+\ell_j$. A sufficiently high Taylor remainder is integrable, together with the parameter derivatives, on any fixed compact parameter set. If a boundary integral has a pole, its polar term contributes to a longer flag. This subtraction is what makes the remaining numerators in \eqref{eq:local-form} holomorphic.

\begin{proposition}[Laurent coefficients from support flags]\label{prop:Laurent}
If the longest incident chain has length $d$, then the order at $q$ is at most $d$, and
\begin{equation}\label{eq:Laurent}
 [(s-q)^{-m}]\xi_\Phi(s)=
 \sum_{|\mathfrak f|\geq m}
 \frac{C_{\mathfrak f}^{(|\mathfrak f|-m)}(q)}
 {(|\mathfrak f|-m)!\prod_{S\in\mathfrak f}N(S)},
 \qquad 1\leq m\leq d.
\end{equation}
\end{proposition}
\begin{proof}
Expand each holomorphic numerator in \eqref{eq:local-form} in powers of $s-q$. Its term of degree $|\mathfrak f|-m$ is the unique term contributing to order $-m$. Summing gives \eqref{eq:Laurent}.
\end{proof}
Two supports that are not nested cannot occur in the same flag. Even if they give the same candidate value, they do not supply a product of denominators in one chart. A multiple pole therefore requires both nesting and a nonzero total coefficient.

\subsection{The coefficient attached to a single support}
For a proper support $S$, write $u=(u_i)_{i\in S}$ and $z=(z_j)_{j\in J}$, and expand
\begin{equation}\label{eq:Taylor-support}
 \prod_{\beta\in R(S)}
 \left(\sum_{i\in S}b_i(\beta)u_i
       +\sum_{j\in J}b_j(\beta)z_j\right)^{-s}
 =\sum_{\nu\in\Z_{\geq0}^{J}}F_{S,\nu}(u;s)z^\nu.
\end{equation}
Here $\nu$ is a multi-index, $|\nu|=\sum_{j\in J}\nu_j$, and $F_{S,\nu}$ includes the Taylor factor $1/\nu!$. Define the weighted complementary zeta function by
\[
 Z_{J,\nu}(s)=\AC\sum_{z\in\N^J}z^\nu
       \prod_{\beta\in\Psi_J^+}
       \left(\sum_{j\in J}b_j(\beta)z_j\right)^{-s}.
\]
The projective coefficient $\cP_{S,\nu}(s)$ is the meromorphic integral of $F_{S,\nu}$ on the simplex $\sum_{i\in S}u_i=1$, with its proper boundary poles subtracted as in \eqref{eq:local-form}. At $q=(|S|-\ell)/N(S)$, provided no longer incident flag contributes, the support contribution is
\begin{equation}\label{eq:single-support-rule}
 \frac1{N(S)}\sum_{|\nu|=\ell}
       \cP_{S,\nu}(q)Z_{J,\nu}(q).
\end{equation}
This is the diagonal form of \cite[Theorem 3.2]{Generic}; $\AC$ denotes meromorphic continuation. For the simple-pole calculations below, the displayed factors are regular at $q$. When a longer flag contributes, the product must instead be evaluated through the full expansion \eqref{eq:local-form}.

\subsection{Integral conventions}\label{sec:integral-conventions}
Let $f(u)$ be homogeneous of degree $-d$ in $d$ positive variables, and let $E=\sum u_i\partial_{u_i}$ be the radial vector field. Its projective integral is the integral of $f\,\iota_E(\dd u_1\wedge\cdots\wedge\dd u_d)$ over a section meeting each positive ray once. Here $\iota_E$ denotes contraction with $E$. Homogeneity makes the choice of section immaterial. On $\sum u_i=1$, we orient the section so that the measure is the positive coordinate measure obtained by eliminating one variable. This is not Euclidean surface measure. A zero-dimensional simplex has measure one.

For an integral, $\AC$ means meromorphic continuation in the exponents of its linear factors, followed by restriction to the stated parameter family. We call such a continued integral a \emph{period}. A \emph{loaded chamber} is an oriented chamber together with a branch of the multivalued density. At nonresonant parameters it can be regularized to a closed twisted cycle. An exact twisted differential integrates to zero on that cycle, and the resulting identity continues meromorphically; see \cite[\S1.1, arXiv version]{Mimachi} for the regularization map and its genericity requirement. At singular parameters we continue the entire identity before evaluating it; we do not multiply separately assigned finite parts.

The full-cube Selberg normalization is
\begin{equation}\label{eq:Selberg}
 S_n(a,b,c)=\prod_{j=0}^{n-1}
 \frac{\Gamma(a+jc)\Gamma(b+jc)\Gamma(1+(j+1)c)}
 {\Gamma(a+b+(n+j-1)c)\Gamma(1+c)},\qquad S_0=1.
\end{equation}
In its convergence region this equals the integral of
$\prod_i x_i^{a-1}(1-x_i)^{b-1}\prod_{i<j}|x_i-x_j|^{2c}$ over $(0,1)^n$. The integral over one ordered chamber is $S_n/n!$. Outside that region \eqref{eq:Selberg} is its meromorphic continuation \cite{FW}.

\subsection{The first normal Taylor coefficient}
\begin{remark}[The first normal coefficient]\label{rem:first-normal}
Choose a Weyl-invariant inner product on the real weight space $V$ and write $V=X_J\oplus V_J$, where $X_J$ is fixed by the parabolic Weyl group $W_J$ and $V_J=X_J^\perp$. For $x\in X_J$ away from the smaller flats, factor the diagonal density as $F_s(x+y)=F_{J,s}(y)G_s(x,y)$, with the normal root factors in $F_{J,s}$. The remaining factor is analytic at $y=0$ and satisfies $G_s(x,wy)=G_s(x,y)$. Its linear coefficient lies in $(V_J^*)^{W_J}=0$, because an invariant linear functional vanishes on the normals of the simple reflections, which span $V_J$.

The original simple-root gaps may differ from these centered coordinates by an exterior translation $u\mapsto u+c(y)$. At degree one the induced difference is $D_cF_0$, and at the critical homogeneity
\[
 D_cF_0\,\iota_E\mathrm{vol}
 =-\dd\bigl(\iota_E(F_0\,\iota_{D_c}\mathrm{vol})\bigr),
 \qquad \mathrm{vol}=\dd u_1\wedge\cdots\wedge\dd u_d.
\]
The independent-exponent Stokes argument of Lemma~\ref{lem:centering} makes this projective period zero. Hence the first normal coefficient also vanishes in \eqref{eq:single-support-rule}; higher degrees and derivatives at multiple intersections are not covered by this argument.
\end{remark}

\section{Type \texorpdfstring{$A$}{A}: support classification}\label{sec:A-count}
For the remainder of the type-$A$ argument, put
\begin{equation}\label{eq:q-main}
 M_r=\binom{r+1}{2}-2,\qquad q_r=q_r^A=\frac{r-4}{M_r},\qquad h_r=\frac{\pi q_r}{2}.
\end{equation}
The positive roots of $A_r$ are intervals. If the connected components of $J=I\setminus S$ have lengths $\lambda_1,\ldots,\lambda_t$, then
\begin{equation}\label{eq:A-count}
 N(S)=\binom{r+1}{2}-\sum_{a=1}^t\binom{\lambda_a+1}{2}.
\end{equation}
The normal subsystem $A_1\sqcup A_1$ with Taylor order two gives $q_r$ in \eqref{eq:q-main}.

\begin{theorem}[Supports at the type-$A$ candidate]\label{thm:complete-A}
Every support satisfying $q(S,\ell)=q_r$ has normal complement $A_1\sqcup A_1$ and Taylor degree two, with the following additional degree-zero complements:
\[
 r=12:\quad J\simeq A_6;\qquad
 r=20:\quad J\simeq A_5.
\]
There are $\binom{r-1}{2}$ supports of the first kind. Rank $12$ has seven additional supports and seventy flags of length two. Rank $20$ has sixteen additional supports and ninety-six flags of length two. No flag at $q_r$ has length three.
\end{theorem}
\begin{proof}
Write $u=r-4$, $M=M_r$, $j=|J|$, and $N_J=|\Psi_J^+|$. The candidate equation is
\begin{equation}\label{eq:integer-candidate}
 \frac{r-j-\ell}{\binom{r+1}{2}-N_J}=\frac{u}{M}.
\end{equation}
Let $g=\gcd(u,M)$. Since $2M=u(u+9)+16$, one has $g\mid16$. The reduced fraction in \eqref{eq:integer-candidate} implies
\[
 \binom{r+1}{2}-N_J=tM/g,
 \qquad r-j-\ell=tu/g
\]
for a positive integer $t$. Here $M/g>2$ for $r\geq5$, so the upper bound $N(S)\leq M+2$ implies $t\leq g$.

If $t=g$, then $N_J=2$. The only type-$A$ parabolic subsystem with two positive roots is $A_1\sqcup A_1$; hence $j=2$ and $\ell=2$.

Suppose $t<g$ and set $x=(g-t)/g$. Then
\[
 N_J=2+xM,\qquad j+\ell=4+xu.
\]
Among type-$A$ parabolic subsystems of rank at most $j$, the largest positive-root count is $j(j+1)/2$. Therefore
\[
 2+\frac{x}{2}\{u(u+9)+16\}\leq
 \frac{(4+xu)(5+xu)}2,
\]
which reduces to $xu^2\leq16$. As $x\geq1/g$ and $g\leq16$, it follows that $u\leq16$. Testing the divisors of $16$ in this finite range leaves only
\[
\begin{array}{c|c|c|c}
r&g&N_J&j\text{ at most}\\\hline
8&2&19&6\\
12&4&21&6\\
20&16&15&5.
\end{array}
\]
The first row is impossible: a rank-six irreducible subsystem has $21$ positive roots, whereas a reducible subsystem of rank at most six has at most $16$. The last two rows force $A_6$ and $A_5$, respectively, with $\ell=0$.

Two nonadjacent nodes in a path of length $r$ can be chosen in $\binom{r-1}{2}$ ways. A block of six consecutive nodes in $A_{12}$ has seven placements and contains ten nonadjacent pairs. A block of five consecutive nodes in $A_{20}$ has sixteen placements and contains six such pairs. These give $70$ and $96$ chains. There are only two support sizes at either exceptional point, which excludes a chain of length three.
\end{proof}

\section{Type \texorpdfstring{$A$}{A}: chamber periods}\label{sec:A-geometry}
The cumulative sums of the simple-root gaps give $r+1$ ordered points. Collapsing two disjoint pairs leaves $m=r-1$ blocks: two contain two points, and the others contain one. We call these \emph{double blocks} and \emph{single blocks}. There are $d=m-1=r-2$ gaps between their centers. For the ordered centers $z_1<\cdots<z_m$ and block sizes $w_i$, put
\[
 D_w(z)=\prod_{i<j}(z_j-z_i)^{w_iw_j}.
\]
The degree of $D_w$ is $M_r$. Passing to the successive center gaps removes translation. If the double blocks have centers $X,Y$, define the rational insertion
\begin{equation}\label{eq:insertion}
 H(z)=\frac4{(Y-X)^2}+
 \sum_{z_i\ne X,Y}\left\{\frac1{(X-z_i)^2}+\frac1{(Y-z_i)^2}\right\}.
\end{equation}
At $s=q_r$, the density $D_w^{-s}H$ has degree $-d$. We therefore integrate it projectively, using the coordinate measure on $\sum v_i=1$ defined in Section~\ref{sec:local}.

The residue formula \eqref{eq:single-support-rule} uses the original simple-root gaps. The next lemma shows that its quadratic coefficient has the same period as the centered expression.

\begin{lemma}[Centered quadratic coefficient]\label{lem:centering}
For a fixed two-pair support, the projective period of the quadratic Taylor coefficient in the original gaps equals that of
\begin{equation}\label{eq:centered-jet}
 \frac{s}{4}D_w^{-s}
 \left(a^2\sum_{B\ne X}\frac{w_B}{(X-z_B)^2}
       +b^2\sum_{B\ne Y}\frac{w_B}{(Y-z_B)^2}\right)
\end{equation}
at $M_rs=d-2$. Here $a,b$ are the gaps inside the double blocks. In each sum, $B$ ranges over the other blocks, with center $z_B$ and size $w_B$. The equality holds for the meromorphic period of each coefficient in $a,b$; in particular, the mixed coefficient $ab$ has zero period.
\end{lemma}
\begin{proof}
Let $u_j$ be the gap between the last point of block $j$ and the first point of block $j+1$. Give the double blocks widths $a,b$ and the single blocks width zero; denote these widths by $\delta_j$. The corresponding center gaps are exactly
\begin{equation}\label{eq:gap-change}
 v_j=u_j+c_j,\qquad c_j=\frac{\delta_j+\delta_{j+1}}2.
\end{equation}
The Jacobian with respect to the exterior gaps is one. We use this change only to compare Taylor coefficients; the original lattice variables remain unchanged.

Let $F(u,a,b;s)$ be the exterior root product to the power $-s$, with the two normal root factors omitted. Let $F^c(v,a,b;s)$ be the same product written with the pairs centered at their midpoints, and put $F_0(v;s)=D_w(v)^{-s}$. Equation~\eqref{eq:gap-change} gives the exact identity
\[
 F(u,a,b;s)=F^c(u+c,a,b;s).
\]
Within each pair, the displacements are $\pm a/2$ or $\pm b/2$. Their linear terms cancel. Expanding the logarithm to degree two and then exponentiating gives \eqref{eq:centered-jet}; the mixed $ab$ term cancels when the four roots between the two pairs are added. Consequently, if $D_c=\sum_jc_j\partial_{u_j}$ and a subscript denotes total Taylor degree, then
\begin{equation}\label{eq:jet-difference}
 F_2-F^c_2=\frac12D_c^2F_0.
\end{equation}

It remains to show that $D_c^2F_0/2$ has zero projective period. We exhibit a primitive. Write $\mathrm{vol}=\dd u_1\wedge\cdots\wedge\dd u_d$ and
\[
 \eta=\frac12(D_cF_0)\,\iota_{D_c}\mathrm{vol}.
\]
Then $\dd\eta=\tfrac12D_c^2F_0\,\mathrm{vol}$. On the critical hyperplane, $F_0$ has degree $2-d$, so $\mathcal L_E\eta=0$. Cartan's formula therefore gives the explicit projective primitive
\begin{equation}\label{eq:projective-primitive}
 \frac12D_c^2F_0\,\iota_E\mathrm{vol}
   =-\dd(\iota_E\eta).
\end{equation}
The identity holds on every branch of the density, hence also in the twisted de Rham complex.

To justify integration of this identity, first vary the exterior exponents independently. Keep them equal on each orbit of the two pair reflections and require their sum to be $d-2$. Give each adjacent-gap form $u_j$ total exponent $-4$, the sum $u_1+\cdots+u_d$ total exponent $5d-2$, and all other collapsed forms exponent zero. These totals are obtained by equal assignments within each root orbit. The resulting density is
\begin{equation}\label{eq:Stokes-witness}
 F_0^*(u)=\frac{\prod_{j=1}^{d}u_j^4}{(u_1+\cdots+u_d)^{5d-2}},
 \qquad \deg F_0^*=2-d.
\end{equation}
On the simplex, $u_1+\cdots+u_d=1$. After two derivatives, every coordinate still occurs to a positive power. The primitive therefore vanishes on every boundary face. The same bounds hold in a sufficiently small relative neighborhood of these exponents, since only finitely many linear forms occur and each is a sum of adjacent gaps. Stokes' theorem gives zero there. Meromorphic continuation along the hyperplane of total exponent $d-2$ gives the coefficientwise identity at the Witten parameters.

Together with \eqref{eq:jet-difference}, this proves the lemma. In the simple-pole cases, Theorem~\ref{thm:complete-A} excludes longer flags at $q_r$. At the exceptional ranks we use the flag formula to compute the leading coefficient.
\end{proof}

Let $J_{i,j}$ be the continued projective integral of $D_w^{-q_r}H$ when the double blocks occupy positions $i<j$. The original normal factors are $a^{-s}b^{-s}$. Summing the two quadratic terms in Lemma~\ref{lem:centering} over $a,b\geq1$ gives $\zeta(s-2)\zeta(s)$ and $\zeta(s)\zeta(s-2)$. The mixed term would give $\zeta(s-1)^2$, but its period is zero by that lemma. The remaining zeta factors are regular at $q_r$. Formula~\eqref{eq:single-support-rule}, whose diagonal denominator has slope $M_r$, gives
\begin{equation}\label{eq:raw-A-residue}
 \Res_{s=q_r}\xi_{A_r}(s)=
 \frac{q_r\zeta(q_r)\zeta(q_r-2)}{4M_r}
 \sum_{1\leq i<j\leq m}J_{i,j},\qquad r\notin\{12,20\}.
\end{equation}
Each summand corresponds to one nonadjacent pair of normal nodes, so no additional multiplicity is required. The double-pole cases are treated separately in Section~\ref{sec:cumulative}.

\subsection{Coordinates with the double blocks at zero and one}
Put $n=r-3$, the number of single blocks. Translate and scale the double-block centers to $0$ and $1$. Let $\mathcal D_{i,j}\subset\R^n$ consist of the ordered single-block coordinates: $i-1$ lie below zero, $j-i-1$ lie between zero and one, and $m-j$ lie above one. All the periods can now be written with the same density:
\begin{equation}\label{eq:fixed-heavy-integral}
 J_{i,j}=\AC\int_{\mathcal D_{i,j}}
 W_{q_r}(x)\left[4+\sum_{a=1}^n\left(\frac1{x_a^2}+\frac1{(1-x_a)^2}\right)\right]\dd x,
\end{equation}
where
\[
 W_s(x)=\prod_{a=1}^n|x_a|^{-2s}|1-x_a|^{-2s}
        \prod_{a<b}|x_a-x_b|^{-s}.
\]
The coordinates are ordered, so \eqref{eq:fixed-heavy-integral} has no factorial prefactor. To check the measure, let $t$ be the separation of the double-block centers before scaling. The coordinate Jacobian is $t^n$, and contraction with the radial vector field gives $t^{n+1}$. The density has the inverse power because $M_rq_r+2=n+1$. Thus the measure agrees with that in \eqref{eq:raw-A-residue}.

\Needspace{21\baselineskip}
\subsection{Removing the quadratic insertion by integration by parts}
Let $I_{i,j}(s)$ be the continued integral of $W_s$ on $\mathcal D_{i,j}$. We show that every inserted period is the same scalar multiple of the corresponding $I_{i,j}$. The key step is the following rational identity.

\begin{proposition}[Quadratic insertion identity]\label{prop:common-insertion}
Put $n=r-3$, $q=q_r$, and
\[
 p=\sum_{a=1}^n\frac1{x_a},\qquad
 \bar p=\sum_{a=1}^n\frac1{1-x_a},\qquad
 \kappa_n=\frac{8(n^2+5n+10)}{5(n^2+11n+4)}.
\]
For $\nabla_iV_i=\partial_iV_i+V_i\partial_i\log W_q$, define
\begin{equation}\label{eq:common-vector}
 V_i=-\frac{6(1-2x_i)
 +4\bigl((1-x_i)/x_i-x_i/(1-x_i)\bigr)
 +(1-x_i)p-x_i\bar p}{5(1+2q)}.
\end{equation}
Then, for $x_i\ne0,1$ and $x_i\ne x_j$ when $i\ne j$,
\begin{equation}\label{eq:common-divergence}
 4+\sum_i\left(\frac1{x_i^2}+\frac1{(1-x_i)^2}\right)-\kappa_n
 =\sum_i\nabla_iV_i.
\end{equation}
Consequently, whenever the continued periods are regular at the displayed parameters,
\begin{equation}\label{eq:common-period}
 J_{i,j}=\kappa_n I_{i,j}(q),\qquad 1\leq i<j\leq r-1.
\end{equation}
\end{proposition}
\begin{proof}
We derive three divergence identities at independent Selberg parameters and then eliminate the reciprocal sums. Write the density as $W=\prod |x_i|^A|1-x_i|^B|\Delta(x)|^{2c}$, where $\Delta(x)=\prod_{i<j}(x_i-x_j)$, put
$h_0=A+B+1+c(n-1)$, and let $p_j=\sum_i x_i^{-j}$. The divergences of the three vector fields with components
\[
 U_i=1-x_i,\qquad Y_i=(1-x_i)/x_i,\qquad T_i=(1-x_i)p_1
\]
are, respectively,
\begin{equation}\label{eq:three-divergences}
\begin{aligned}
 \sum_i\nabla_iU_i&=Ap_1-nh_0,\\
 \sum_i\nabla_iY_i&=(A-1+c)p_2-(A+B)p_1-cp_1^2,\\
 \sum_i\nabla_iT_i&=Ap_1^2-p_2-(nh_0-1)p_1.
\end{aligned}
\end{equation}
For example, the interaction in the second line is
$-2c\sum_{i<j}(x_ix_j)^{-1}=-c(p_1^2-p_2)$; the other pair terms follow by the same two-term grouping.

Set
\[
 \alpha=\frac{A(A+B)+c(nh_0-1)}{A(A-1)(A+c)},\qquad
 \beta=\frac{A}{(A-1)(A+c)},\qquad
 \gamma=\frac{c}{(A-1)(A+c)}.
\]
Multiplying the three lines of \eqref{eq:three-divergences} by $\alpha,\beta,\gamma$ and adding gives
\[
 \sum_i\nabla_i(\alpha U_i+\beta Y_i+\gamma T_i)
 =p_2-\alpha nh_0.
\]
For $A=B=-2q$, $c=-q/2$, and
$q=2(n-1)/(n^2+7n+8)$, the coefficients become
\[
 (\alpha,\beta,\gamma)=-\frac{(6,4,1)}{5(1+2q)}.
\]
Apply the generic identity a second time with $x_i$ replaced by $1-x_i$ and $A,B$ interchanged. The reflected vector field has an additional minus sign from differentiation. At $A=B=-2q$ the two fields sum to \eqref{eq:common-vector}, and their divergences add to $\sum_i\{x_i^{-2}+(1-x_i)^{-2}\}-2\alpha nh_0$. Finally
\[
 4+2\alpha nh_0=\kappa_n.
\]
This proves \eqref{eq:common-vector}--\eqref{eq:common-divergence}.

These rational identities hold at independent parameters. Integrate them on a regularized closed chamber, where each divergence has zero period, and then specialize the resulting meromorphic identity. The equality $A=B$ is used only at the final step. This proves \eqref{eq:common-period} without requiring a common convergence region for the unbounded real integrals.
\end{proof}

\subsection{Relations between the chamber integrals}
For $k\geq0$, define
\[
 [k]_h!=\prod_{\nu=1}^{k}\sin(\nu h),\qquad [0]_h!=1.
\]
\begin{lemma}[Relations among the type-$A$ chamber periods]\label{lem:chamber-ratios}
For $r\notin\{12,20\}$ and $h=h_r$, the two-pair chamber relations are
\begin{align}\label{eq:omega}
 J_{i,j}&=\omega_{i,j}^{(m)}(h)J_{1,2},\notag\\
 \omega_{i,j}^{(m)}(h)&=(-1)^{j-i-1}
 \frac{[3]_h![m+j-i+3]_h![m-2]_h![m+1]_h!}
 {[m+4]_h![j-i-1]_h![m-j]_h![m-i+2]_h![j+1]_h![i-1]_h!}.
\end{align}
The finite expressions $\omega_{i,j}^{(m)}$ are also defined at $r=12,20$, but the displayed period identity is used only outside those ranks. The double-pole coefficients are computed from flags in Section~\ref{sec:cumulative}.
\end{lemma}
\begin{proof}
By Proposition~\ref{prop:common-insertion}, it suffices to prove the relations for the uninserted periods $I_{i,j}$. Use the upper- and lower-half-plane contour relations for the Selberg density, as in \cite[equations (2.31)--(2.33), arXiv numbering]{FW}. These are relations between regularized loaded chambers. They hold at nonresonant parameters and continue meromorphically after their regularization denominators are cleared.

Remove one labeled single block. Suppose the double blocks occupy positions $a<b$ in the remaining ordered sequence, and put $T=e^{ih}$. Moving one single block past another gives phase $T^{-2}$; moving it past a double block gives phase $T^{-4}$. Insert the removed block before, between, or after the double blocks. The contour relation is
\[
 C_0^-I_{a+1,b+1}+C_1^-I_{a,b+1}+C_2^-I_{a,b}=0,
\]
where summing the phases over the possible positions of the single block gives:
\begin{align*}
 C_0^-&=T^{-(a-1)}\frac{\sin(ah)}{\sin h},&
 C_1^-&=T^{-(a+b+1)}\frac{\sin((b-a)h)}{\sin h},&
 C_2^-&=T^{-(b+m+3)}\frac{\sin((m-b)h)}{\sin h}.
\end{align*}
The lower-half-plane relation is the conjugate relation. Their real kernel has coordinate ratios
\[
 -\frac{\sin((m-a+2)h)}{\sin(ah)}:
 \frac{\sin((m+b-a+4)h)}{\sin((b-a)h)}:
 -\frac{\sin((b+2)h)}{\sin((m-b)h)}.
\]
The two recurrences connect adjacent positions in the array $1\leq i<j\leq m$. Starting with $(1,2)$ determines each row and then the next row. Substitution verifies that the displayed $\omega_{i,j}^{(m)}$, with $\omega_{1,2}^{(m)}=1$, satisfies both recurrences. This proves \eqref{eq:omega}. Its denominator factors have indices at most $r+3$, and
\[
 (r+3)q_r<2.
\]
They are consequently nonzero at the stated points. The quantity $\Omega_r$ in \eqref{eq:Sigma-Omega} is the coefficient relating the positive reference chamber $J_{1,r-1}$ to $J_{1,2}$; it is nonzero outside ranks $12$ and $20$, as checked in Section~\ref{sec:A-evaluation}. Multiplication by the common $\kappa_n$ returns the relations for $J_{i,j}$.
\end{proof}
Write
\begin{equation}\label{eq:Sigma-Omega}
 \Sigma_r(h)=\sum_{i<j}\omega_{i,j}^{(r-1)}(h),\qquad
 \Omega_r(h)=\omega_{1,r-1}^{(r-1)}(h)
 =(-1)^{r-3}\frac{[2r]_h![3]_h!}{[r+3]_h![r]_h!}.
\end{equation}
Removing the insertion in the chamber $J_{1,r-1}$, where all single blocks lie between the double blocks, gives a convergent reference integral.

\section{Type \texorpdfstring{$A$}{A}: residue evaluation}\label{sec:A-evaluation}
\subsection{A finite sine identity}
\begin{theorem}[Sum of the type-$A$ chamber weights]\label{thm:sine}
For $r\geq5$ and $h=h_r$,
\begin{equation}\label{eq:Sigma-product}
 \Sigma_r(h)=
 -2^{r-2}\frac{[r-3]_h!\sin^3(2h)\sin(3h)}
 {\sin((r+1)h)\sin((r+2)h)\sin((r+3)h)}<0.
\end{equation}
\end{theorem}
\begin{proof}
The proof has three steps: rewrite the double chamber sum as a convolution, encode that convolution in a Laurent polynomial, and apply a finite difference operator that leaves only four terms. Put $t=e^{ih}$ and
\[
 B_r(z)=\sum_{a=0}^r\frac{[r]_h!}{[a]_h![r-a]_h!}z^a
       =\prod_{\nu=0}^{r-1}(1+zt^{2\nu-r+1}).
\]
The product formula is the finite binomial theorem after symmetrizing its phases. Define
\[
 C_k=\sum_{a=0}^{r-k}
 \frac{[r]_h!}{[a]_h![r-a]_h!}
 \frac{[r]_h!}{[a+k]_h![r-a-k]_h!}.
\]
In the sum of \eqref{eq:omega}, put $a=i-1$ and $k=j-i+2$. This gives the explicit convolution
\begin{equation}\label{eq:sine-convolution}
 \Sigma_r(h)=\frac{[3]_h![r-3]_h!}{[r+3]_h![r]_h!}
 \sum_{k=3}^r(-1)^{k-3}\frac{[r+k]_h!}{[k-3]_h!}C_k.
\end{equation}
Let $\tau f(z)=f(tz)$ and $R(z)=B_r(z)B_r(z^{-1})$. The coefficient of $z^k$ in $R$ is $C_{|k|}$. Each operator factor below multiplies $z^k$ by $\sin((k+\ell)h)$. The product therefore annihilates every monomial with $-r\leq k\leq2$. Evaluating the remaining monomials at $z=-1$ rewrites \eqref{eq:sine-convolution} as
\begin{equation}\label{eq:finite-operator}
 \Sigma_r(h)=-\frac{[3]_h![r-3]_h!}{[r+3]_h![r]_h!}
 \left[\prod_{\ell=-2}^{r}
 \frac{t^{\ell}\tau-t^{-\ell}\tau^{-1}}{2i}\,R\right](-1).
\end{equation}
All sums and products are finite, so this rearrangement requires no convergence argument.

Set $Q=t^2$ and use $(a;Q)_p=\prod_{j=0}^{p-1}(1-aQ^j)$, with $(a;Q)_0=1$. The coefficient of $\tau^{2p-r-3}$ in the operator is
\[
 d_p=\frac{(-1)^{r+3+p}t^{-U+p(r-2)}}{(2i)^{r+3}}
       \frac{[r+3]_h!}{[p]_h![r+3-p]_h!},
 \qquad U=\frac{(r-2)(r+3)}2.
\]
For $2\leq p\leq r+1$,
\[
 B_r(-t^{2p-r-3})=(Q^{p-r-1};Q)_r=0.
\]
Only $p=0,1,r+2,r+3$ remain. The two paired values of $R$ are
\[
 4^r\left(\frac{[r+1]_h!}{\sin h}\right)^2,
 \qquad 4^r([r]_h!)^2.
\]
Using $(r^2+r-4)h=\pi(r-4)$, their paired coefficients simplify to
\[
 d_0+d_{r+3}=2^{-r-2}\sin h,
 \qquad d_1+d_{r+2}=-2^{-r-2}
 \frac{\sin((r+3)h)\sin((r-1)h)}{\sin h}.
\]
Substitution in \eqref{eq:finite-operator}, followed by
\[
 \sin^2((r+1)h)-\sin((r+3)h)\sin((r-1)h)=\sin^2(2h),
\]
gives \eqref{eq:Sigma-product}. All its sine factors have arguments in $(0,\pi)$, so the sign is negative.
\end{proof}

\subsection{Evaluation of the reference integral}
\begin{proposition}[The reference chamber integral]\label{prop:endpoint}
Put $n=r-3$. The inserted period with all single blocks between the double blocks is
\begin{equation}\label{eq:Jend}
 J_{1,r-1}=\frac{\kappa_n}{n!}
 S_n(1-2q_r,1-2q_r,-q_r/2)>0.
\end{equation}
\end{proposition}
\begin{proof}
In this chamber $0<x_1<\cdots<x_n<1$. Removing the insertion leaves the Selberg integral $S_n/n!$. Proposition~\ref{prop:common-insertion} gives \eqref{eq:Jend}; the inserted integral is interpreted by meromorphic continuation.

To check ordinary convergence of the uninserted integral, set
$a=b=1-2q_r$, $c=-q_r/2$, and $D=n^2+7n+8$. Then
\[
 a+(n-1)c=b+(n-1)c=\frac{5n+11}{D}>0,
 \qquad 1+nc=\frac{8(n+1)}D>0.
\]
For $c<0$ these inequalities imply all endpoint and interior collision conditions for Selberg's integral. Its density is positive in the interior. Since $\kappa_n>0$, the continued inserted period is positive as well.
\end{proof}

\subsection{Simple poles and the exceptional ranks}
When $\Omega_r(h_r)\ne0$, equation~\eqref{eq:omega} gives
\[
 \sum_{i<j}J_{i,j}=\frac{\Sigma_r(h_r)}{\Omega_r(h_r)}J_{1,r-1}.
\]
Combining this identity with \eqref{eq:raw-A-residue} and \eqref{eq:Jend} gives
\begin{equation}\label{eq:A-residue}
 \boxed{\displaystyle
 R_r:=\Res_{s=q_r}\xi_{A_r}(s)=
 \frac{q_r\zeta(q_r)\zeta(q_r-2)}{4M_r}
 \frac{\Sigma_r(h_r)}{\Omega_r(h_r)}
 \frac{\kappa_{r-3}}{(r-3)!}
 S_{r-3}(1-2q_r,1-2q_r,-q_r/2).}
\end{equation}
The two zeta factors are negative for $0<q_r<1$, so their product is positive. The endpoint and $\Sigma_r$ are nonzero. A zero of $\Omega_r$ can occur only when $kq_r=2$ for $r+4\leq k\leq2r$. Since
\[
 \frac{2}{q_r}=r+5+\frac{16}{r-4},
\]
the only valid ranks are $r=12$ and $r=20$; the vanishing sine factors have indices $19$ and $26$. Theorem~\ref{thm:complete-A} then proves the simple-pole assertion in every other rank.

At ranks $12$ and $20$, formula~\eqref{eq:A-residue} has a zero denominator and cannot be specialized. We compute the double-pole coefficients from the flags in Section~\ref{sec:cumulative}.

\subsection{Face integrals for a connected normal complement}\label{sec:cumulative}
The double-pole coefficients in \eqref{eq:A12-main}--\eqref{eq:A20-main} require the following convergent face integrals. We evaluate them before applying the two-step residue formula.

Let $R>L\geq1$, put $k=R-L$, and in this subsection take $I=\{1,\ldots,R\}$. Define
\[
 J_{L,j}=\{j,j+1,\ldots,j+L-1\}\subseteq\{1,\ldots,R\},
 \qquad U_{L,j}=I\setminus J_{L,j},\quad 1\leq j\leq k+1.
\]
For $u_i>0$ with $\sum_{i\in U_{L,j}}u_i=1$, set
\begin{align}
 Q_{R,L,j}(u)&=
 \prod_{\substack{1\leq a\leq b\leq R\\{}[a,b]\not\subseteq J_{L,j}}}
 \left(\sum_{i\in[a,b]\cap U_{L,j}}u_i\right),\label{eq:outer-Q}\\
 \cP_{R,L,j}(s)&=\int_{\sum u_i=1}Q_{R,L,j}(u)^{-s}\dd u,
 \qquad \cP_{R,L}^{\mathrm{sum}}=\sum_{j=1}^{k+1}\cP_{R,L,j}.
 \label{eq:outer-period}
\end{align}
The simplex measure is the one obtained by eliminating one gap coordinate; a zero-dimensional simplex has measure one. The number of factors in \eqref{eq:outer-Q} is
\[
 n_{R,L}=\binom{R+1}{2}-\binom{L+1}{2}
 =\frac{(R-L)(R+L+1)}2.
\]
The critical exponent is consequently $q=2/(R+L+1)$. For $T\subseteq I$, let $N_R(T)$ denote the number of positive roots of $A_R$ whose support meets $T$.

\begin{lemma}[Convergence of the outer face integrals]\label{lem:exposure}
At this $q$, every integral $\cP_{R,L,j}(q)$ is ordinarily convergent and strictly positive.
\end{lemma}
\begin{proof}
For a proper nonempty $T\subsetneq U_{L,j}$, let $c=R-|T|>L$ be the size of its complement. The complement has at most $\binom{c+1}{2}$ positive roots. Hence
\[
 \frac{|T|}{N_R(T)}\leq\frac{2}{R+c+1}<\frac{2}{R+L+1}=q.
\]
Indeed, let $V=U_{L,j}\setminus T$ be the coordinates tending to zero at that boundary. Exactly $N_R(U_{L,j})-N_R(T)$ restricted root factors vanish there. Since $qN_R(U_{L,j})=|U_{L,j}|$, the displayed inequality is equivalent to $q\{N_R(U_{L,j})-N_R(T)\}<|V|$. This is the boundary integrability condition. Checking every proper $T$ also checks nested boundary faces, for instance by ordered gap sectors. The density is strictly positive in the interior.
\end{proof}

\begin{theorem}[Explicit sum of connected-complement face periods]\label{thm:outer-sum}
Let $q=2/(R+L+1)$, $h=\pi q/2$, and $k=R-L$. Then
\begin{equation}\label{eq:outer-individual}
 \cP_{R,L,j}(q)=\sbinom{k}{j-1}E_{R,L},
 \qquad
 E_{R,L}=\frac{S_{k-1}(1-q,(k-1)q/2,-q/2)}{(k-1)!}.
\end{equation}
In particular,
\begin{equation}\label{eq:outer-sum}
 \boxed{\displaystyle
 \cP_{R,L}^{\mathrm{sum}}(q)=
 \frac{2^k S_{k-1}(1-q,(k-1)q/2,-q/2)}{(k-1)!}
 \prod_{\nu=0}^{k-1}\cos\!\left(\frac{(2\nu-k+1)\pi q}{4}\right).}
\end{equation}
Here $S_0=1$ and $\sbinom{k}{a}=[k]_h!/([a]_h![k-a]_h!)$. Every displayed factor in the period sum is positive.
\end{theorem}
\begin{proof}
The collapsed block contains $L+1$ points; each of the other $k$ blocks contains one. First place the collapsed block at the left endpoint, fix its center at $0$, and fix the rightmost single block at $1$. This coordinate choice gives the simplex measure in \eqref{eq:outer-period}. The other $k-1$ single-block centers have density
\[
 \prod_i x_i^{-(L+1)q}(1-x_i)^{-q}
 \prod_{i<j}(x_j-x_i)^{-q},\qquad 0<x_1<\cdots<x_{k-1}<1.
\]
Since $1-(L+1)q=(k-1)q/2$, symmetry of the Selberg product proves the endpoint value in \eqref{eq:outer-individual}.

We next compare the possible positions of the collapsed block. Write $n=k-1$ and
\[
 a=1-q,\qquad b=nq/2,\qquad c=-q/2.
\]
Fix the collapsed-block center at zero. Unless it is the rightmost block, fix the rightmost single-block center at one. Replacing each remaining coordinate by $1-x$ gives a mixed Selberg chamber with $p$ variables in $(0,1)$ and $n-p$ in $(1,\infty)$. Normalize its integral by $p!(n-p)!$. The mixed-chamber recurrence \cite[(2.33), arXiv numbering]{FW} is
\[
 \frac{K_p}{K_{p-1}}=
 \frac{\sin\pi(n-p+1)c\,\sin\pi(a+b+(n+p-2)c)}
 {\sin\pi pc\,\sin\pi(a+(p-1)c)}
 =\frac{\sin((n-p+1)h)}{\sin((p+1)h)}.
\]
Successive positions of the collapsed block therefore have the ratios of
$\sbinom{k}{j-1}$. Reflection supplies the final endpoint, whose period equals the first. This proves \eqref{eq:outer-individual}. The recurrence is a meromorphic identity; at the present parameters all the actual face integrals converge by Lemma~\ref{lem:exposure} and no sine denominator vanishes.

To sum the ratios, put $Q=e^{2ih}$. The finite $q$-binomial identity gives
\[
 \sum_{a=0}^{k}\sbinom{k}{a}
 =\prod_{\nu=0}^{k-1}\bigl(1+e^{i(2\nu-k+1)h}\bigr)
 =2^k\prod_{\nu=0}^{k-1}\cos((2\nu-k+1)h/2).
\]
Indeed, the coefficient of degree $a$ in the first product is
$e^{i(a^2-ak)h}\genfrac{[}{]}{0pt}{}{k}{a}_{Q}=\sbinom{k}{a}$, so no infinite-series identity at a root of unity is involved. This proves \eqref{eq:outer-sum}.

Finally, $kh<\pi$ and $|(2\nu-k+1)h/2|<\pi/2$. The Selberg integral is positive and convergent: its most restrictive endpoint parameter is $b+(n-1)c=q/2>0$, and its interior collision bound is $1+nc>0$. The case $k=1$ gives two point periods, as required.
\end{proof}

The case $L=1$ recovers the second-wall sum in \cite{Walls}. Here we need the connected complements $A_6$ and $A_5$ that occur in ranks $12$ and $20$.

\Needspace{15\baselineskip}
\subsection{A two-step residue formula}
The relative-face formula \cite[Theorems 5.3--5.4]{Generic} gives the following coefficient. The hypothesis is an upper bound on the pole order, so the complementary residue is allowed to vanish.

\begin{lemma}[Coefficient of a two-step flag]\label{lem:two-step}
Let $U$ be a degree-zero outer support with $N(U)=n_0$ and connected normal complement $\Phi_J$. Assume its face period $\cP_U(q)$ converges. Suppose $\xi_{\Phi_J}$ has pole order at most one at $q$, and all its possibly contributing supports $T$ have the same root count $p_0$. Write $c_T$ for the numerator attached to $T$. Assume the relative-face formula gives $\cP_U(q)c_T$ as the numerator for the corresponding ambient flag. Then the sum of the length-two coefficients over $U$ is
\begin{equation}\label{eq:cumulative-simple}
 a_{-2}^{(U)}=\frac{p_0}{n_0(n_0+p_0)}
     \cP_U(q)\Res_{s=q}\xi_{\Phi_J}(s).
\end{equation}
The identity holds also when $\Res_q\xi_{\Phi_J}=0$. It gives the full double-pole coefficient after summing over all outer faces, provided no other flag of length at least two contributes.
\end{lemma}
\begin{proof}
Within the complementary subsystem, the coefficient attached to $T$ is $c_T/p_0$. Thus
$\Res_q\xi_{\Phi_J}=\sum_T c_T/p_0$; individual terms may cancel. The corresponding ambient flag is $U\subsetneq U\cup T$. Its root counts are $n_0$ and $n_0+p_0$: the second count still includes all roots meeting $U$. The relative-face residue numerator is $\cP_U(q)c_T$. Hence its diagonal coefficient is
\[
 \frac{\cP_U(q)c_T}{n_0(n_0+p_0)}.
\]
Summing this equality proves \eqref{eq:cumulative-simple}, without dividing by the complementary residue. In the applications below the outer Taylor order is zero: its residue factors as the ordinary outer period times the complementary root-system function. Taking a residue in that complementary function gives the numerator factorization assumed in the lemma.
\end{proof}

\begin{proof}[Completion of the proof of Theorem~\ref{thm:main}]
For rank $12$, the complement is $A_6$. The root counts are $n_0=78-21=57$ and $p_0=21-2=19$. Lemma~\ref{lem:two-step} gives the multiplier $1/228$.

For rank $20$, the complement is $A_5$. Here $n_0=210-15=195$ and $p_0=15-2=13$, so the multiplier is $1/3120$.

Theorem~\ref{thm:complete-A} lists all flags at these points. Equation~\eqref{eq:cumulative-simple} therefore gives \eqref{eq:A12-main} and \eqref{eq:A20-main}, and Theorem~\ref{thm:outer-sum} evaluates their outer periods.

Formula~\eqref{eq:A-residue} gives $R_6>0$ and $R_5<0$. Since the outer periods are positive, both double-pole coefficients are nonzero. No flag has length three, so both poles have order exactly two.
\end{proof}

\section{Chamber sums for singleton complements}\label{sec:cancellation}
We now evaluate the chamber sums associated with one normal coordinate in types $B$, $C$, and $D$. Each sum factors into a reference Selberg integral and a finite connection sum. A zero of the second factor may cancel a pole of the first. We compute both orders before evaluating their product; Section~\ref{sec:actual-singletons} then identifies the resulting Witten coefficients.

\subsection{Mixed Selberg integrals and the reference period}
Put $n=r-2$, and let $\varepsilon=1$ for types $B,C$ and $\varepsilon=0$ for type $D$. The Selberg parameters for this calculation are
\[
 a_\varepsilon(s)=\frac{1-\varepsilon s}{2},\qquad
 b(s)=1-2s,\qquad c(s)=-s/2.
\]
For $0\leq p\leq n$, define
\begin{equation}\label{eq:mixed-cube}
 K_p(s)=\frac1{p!(n-p)!}\AC\int_{(0,1)^p\times(1,\infty)^{n-p}}
 \prod_i x_i^{a_\varepsilon(s)-1}|1-x_i|^{b(s)-1}
 \prod_{i<j}|x_i-x_j|^{2c(s)}\dd x.
\end{equation}
Thus $K_p$ is the integral over one chamber with the coordinates ordered separately in the two intervals. We suppress $\varepsilon$ in $K_p$; it is always the same as in $F_\varepsilon$. Let $\widetilde K_p=p!(n-p)!K_p$. In \cite[(2.33), arXiv numbering]{FW}, the recurrence for $\widetilde K_p/\widetilde K_{p-1}$ contains the factor $p/(n-p+1)$. Since
\[
 \frac{K_p}{K_{p-1}}=\frac{n-p+1}{p}\,
 \frac{\widetilde K_p}{\widetilde K_{p-1}},
\]
that combinatorial factor cancels in our ordered-chamber normalization. The recurrence is therefore
\begin{equation}\label{eq:singleton-ratio}
 \rho_p(s):=\frac{K_p(s)}{K_{p-1}(s)}
 =-\frac{\sin((n-p+1)\theta)\cos((n+p+2+\varepsilon)\theta)}
 {\sin(p\theta)\cos((p-1+\varepsilon)\theta)},\qquad \theta=\pi s/2.
\end{equation}
Use $K_0$, the chamber with all variables in $(1,\infty)$, as the reference period. Substituting $x_i\mapsto1/x_i$ gives
\begin{equation}\label{eq:singleton-endpoint}
 K_0(s)=\frac1{n!}
 S_n\!\left(-\frac12+(n+1+\varepsilon/2)s,1-2s,-s/2\right).
\end{equation}
The sum over these chambers is therefore the meromorphic product
\begin{equation}\label{eq:KF}
 \sum_{p=0}^{n}K_p(s)=K_0(s)F_\varepsilon(s),
 \qquad F_\varepsilon(s)=\sum_{p=0}^{n}\prod_{j=1}^{p}\rho_j(s).
\end{equation}
We call $F_\varepsilon$ the finite connection factor. The subscript records whether the parameters are those of type $D$ ($\varepsilon=0$) or of types $B,C$ ($\varepsilon=1$).

\subsection{Evaluation of the normalized \texorpdfstring{$B/C$}{B/C} sum}
Evaluating $F_1$ before specializing $s$ will let us determine its order of vanishing.

\begin{proposition}[Closed form of the $B/C$ connection factor]\label{prop:BC-meromorphic-product}
For $n\geq2$, set $M=(n+1)(n+3)$ and $\theta=\pi s/2$. The function $F_1$ in \eqref{eq:KF} satisfies
\begin{equation}\label{eq:BC-meromorphic-product}
 F_1(s)=\frac12\cos(M\theta-n\pi/2)\,
 \frac{\sin((n+2)\theta)\prod_{j=2}^{n}\sin(j\theta)}
 {\cos((n+3)\theta)\prod_{j=1}^{n+1}\cos(j\theta)}.
\end{equation}
The equality is meromorphic in $s$.
\end{proposition}
\begin{proof}
Put $Q=e^{i\pi s}$ and write $(a;Q)_p=\prod_{j=0}^{p-1}(1-aQ^j)$, with $(a;Q)_0=1$. Define
\[
 f(Z)=\sum_{p=0}^{n}
 \frac{(Q^{-n};Q)_p(-Q^{n+4};Q)_p}{(Q;Q)_p(-Q;Q)_p}Z^p,
 \qquad F_1(s)=f(Q^{-1}).
\]
The finite $q$-Chu--Vandermonde identities \cite[17.6.2--17.6.3]{DLMF} give
\[
 f(Q^{-3})=P,\qquad f(Q)=BP,
 \quad P=\frac{(Q^{-n-3};Q)_n}{(-Q;Q)_n},
 \quad B=(-1)^nQ^{n(n+4)}.
\]
No restriction $|Q|<1$ is required for these rational identities after termination. With $D=Q^{-n}-Q^{n+4}$, coefficient comparison in the finite polynomial gives
\[
 Q(Q^3Z-1)f(QZ)+DZf(Z)+(Q-Z)f(Z/Q)=0.
\]
\begin{samepage}
Evaluate this identity at $Z=Q^{-2},Q^{-1},1$. Dividing by $P$ and writing $x=f(Q^{-1})/P$, $y=f(Q^{-2})/P$, $z=f(1)/P$, the three equations are
\begin{align*}
 Q(Q-1)x+DQ^{-2}y+Q-Q^{-2}&=0,\\
 Q(Q^2-1)z+DQ^{-1}x+(Q-Q^{-1})y&=0,\\
 Q(Q^3-1)B+Dz+(Q-1)x&=0.
\end{align*}
\end{samepage}
Eliminating $y,z$ gives
\[
 x=\frac{(1-Q^2)(1-Q^3)(1+BQ^3)}{D^2-(Q^3-Q)^2}.
\]
Since
\[
 D^2-(Q^3-Q)^2=Q^{-2n}(1-Q^{2n+2})(1-Q^{2n+6}),
\]
we obtain
\begin{equation}\label{eq:BC-Q-product}
 F_1(s)=\frac{Q^{2n}(1-Q^2)(1-Q^3)}
 {(1-Q^{2n+2})(1-Q^{2n+6})}
 \frac{(Q^{-n-3};Q)_n}{(-Q;Q)_n}
 \bigl(1+(-1)^nQ^M\bigr).
\end{equation}
Converting the finite factors to sines and cosines gives \eqref{eq:BC-meromorphic-product}. The derivation holds where the denominators are nonzero and extends elsewhere by meromorphic continuation.
\end{proof}

Let $\ell=2m$ satisfy $2\leq\ell<n+1$. For each $\varepsilon\in\{0,1\}$ set
\[
 M_\varepsilon=n(n+3+\varepsilon)+1+2\varepsilon,
 \qquad q_\varepsilon=\frac{n+1-\ell}{M_\varepsilon},\qquad \theta_\varepsilon=\pi q_\varepsilon/2.
\]
\begin{theorem}[Connection factors at the shifted values]\label{thm:singleton}
The two families satisfy
\begin{align}
 F_1(q_1)&=0,\label{eq:BC-zero}\\
 F_0(q_0)&=\frac{(-1)^{m+1}}4
 \frac{\sin(2\theta_0)\prod_{j=1}^{n}\sin(j\theta_0)}
 {\prod_{j=1}^{n+2}\cos(j\theta_0)}\ne0.\label{eq:D-nonzero}
\end{align}
If the $B/C$ reference period $K_0$ is holomorphic at $q_1$, then \eqref{eq:BC-zero} makes the chamber sum zero. If $K_0$ has a pole there, one must add the orders of the two factors in \eqref{eq:KF}.
\end{theorem}
\begin{proof}
Fix $\varepsilon$, write $q=q_\varepsilon$, and put $Q=e^{i\pi q}$. Equation~\eqref{eq:singleton-ratio} gives
\[
 F_\varepsilon(q)=f_\varepsilon(Q^{-1}),\qquad
 f_\varepsilon(Z)=\sum_{p=0}^{n}
 \frac{(Q^{-n};Q)_p(-Q^{n+3+\varepsilon};Q)_p}
 {(Q;Q)_p(-Q^\varepsilon;Q)_p}Z^p.
\]
Coefficient comparison gives the polynomial identity
\begin{align}
 &Q^\varepsilon(Q^3Z-1)f_\varepsilon(QZ)
 +\{Q^\varepsilon-Q+(Q^{-n}-Q^{n+3+\varepsilon})Z\}f_\varepsilon(Z)
 \notag\\*[-1mm]
 &\hspace{40mm}+(Q-Z)f_\varepsilon(Z/Q)=0.\label{eq:singleton-qdiff}
\end{align}
The terminating $q$-Chu--Vandermonde evaluations, together with $Q^{M_\varepsilon}=(-1)^{n+1}$, give
\[
 f_\varepsilon(Q^{-3})=P,\qquad
 f_\varepsilon(Q)=-Q^{-(1+2\varepsilon)}P,
 \qquad P=\frac{(Q^{-n-3};Q)_n}{(-Q^\varepsilon;Q)_n}.
\]
Substitution of $Z=Q^{-2},Q^{-1},1$ in \eqref{eq:singleton-qdiff} and elimination of the two intermediate values gives $f_1(Q^{-1})=0$ and
\[
 f_0(Q^{-1})=
 (-1)^nQ^{-n(n+1)/2}
 \frac{(Q;Q)_n(1-Q^2)}{(-1;Q)_{n+3}}.
\]
Converting each factor to a sine or cosine proves \eqref{eq:D-nonzero}. All manipulations are finite. The denominators are nonzero because $0<q$ and $(n+2)q<1$ in the specified range; the latter follows from $(n+2)(n-1)<M_\varepsilon$. The same inequalities prove the nonvanishing in \eqref{eq:D-nonzero}.

The final assertion follows by multiplying the Laurent expansions of $K_0$ and $F_\varepsilon$.
\end{proof}

\subsection{When the \texorpdfstring{$B/C$}{B/C} chamber sum is nonzero}
We next combine the zero of $F_1$ with the poles of the reference period $K_0$. The conclusion concerns their chamber sum. The corresponding Witten residue also requires the quadratic insertion and the other supports, treated in Section~\ref{sec:actual-singletons}.

\begin{theorem}[Nonzero values of the $B/C$ chamber sum]\label{thm:BC-complete-endpoint}
Let $n\geq2$, let $\ell=2m$ satisfy $2\leq\ell<n+1$, and put
\[
 M=(n+1)(n+3),\qquad q=\frac{n+1-\ell}{M},\qquad k=\frac1q.
\]
The connection factor $F_1$ has a simple zero at $q$. Here $K_0$ is the $B/C$ reference period, so $\varepsilon=1$ in \eqref{eq:singleton-endpoint}. The meromorphic chamber sum
$\mathcal K_n(s)=K_0(s)F_1(s)$ is holomorphic at $q$, and
\begin{equation}\label{eq:BC-exception-criterion}
 \mathcal K_n(q)\ne0
 \quad\Longleftrightarrow\quad
 k\in\Z\ \text{and}\ n+4\leq k\leq2n+3.
\end{equation}
If the condition holds, put $j_0=2n+3-k$. Then
\begin{align}
 \mathcal K_n(q)&=\bigl(\Res_{s=q}K_0(s)\bigr)F_1'(q),
 \label{eq:BC-exception-value}\\
 \sgn\mathcal K_n(q)&=(-1)^{n-j_0+m}.
 \label{eq:BC-exception-sign}
\end{align}
In every other case $\mathcal K_n$ has a simple zero at $q$.
\end{theorem}
\begin{proof}
Write
\[
 A_n(\theta)=
 \frac{\sin((n+2)\theta)\prod_{j=2}^{n}\sin(j\theta)}
 {\cos((n+3)\theta)\prod_{j=1}^{n+1}\cos(j\theta)}.
\]
Since $(n+3)q\leq(n-1)/(n+1)<1$, all its displayed factors are finite and positive at $\theta_q=\pi q/2$. Moreover,
\[
 M\theta_q-n\pi/2=(1-\ell)\pi/2.
\]
Proposition~\ref{prop:BC-meromorphic-product} therefore gives the exact derivative
\begin{equation}\label{eq:BC-derivative}
 F_1'(q)=(-1)^{m+1}\frac{\pi M}{4}A_n(\theta_q)\ne0.
\end{equation}

The first numerator gamma arguments of the endpoint \eqref{eq:singleton-endpoint} are
\[
 a_j(q)=-\frac12+\frac{2n+3-j}{2}q,\qquad 0\leq j\leq n-1.
\]
They lie strictly between $-1/2$ and $1/2$. All other numerator and denominator arguments are positive: they are
\[
 1-\frac{j+4}{2}q,\quad 1-\frac{j+1}{2}q,\quad
 \frac12+\frac{n-j}{2}q,\quad 1-\frac q2.
\]
Thus $K_0$ has no zero and is regular unless exactly one $a_j(q)$ is zero. This happens precisely when $k=2n+3-j$ is an integer in $[n+4,2n+3]$. In that case $K_0$ has a simple pole, so it cancels the simple zero of $F_1$, giving \eqref{eq:BC-exception-value}. Among its other first gamma arguments exactly $n-1-j_0$ are negative, each in $(-1,0)$; its residue has sign $(-1)^{n-1-j_0}$. Combining with \eqref{eq:BC-derivative} proves \eqref{eq:BC-exception-sign}. If the endpoint is regular, both it and $A_n(\theta_q)$ are nonzero, so the zero of the product is simple.
\end{proof}

For the exceptional values, the residue of the single singular gamma factor gives
\begin{equation}\label{eq:BC-endpoint-residue}
 \Res_{s=q}K_0(s)=\frac{2}{k\,n!}
 \frac{\displaystyle
 \prod_{\substack{0\leq j\leq n-1\\j\ne j_0}}
       \Gamma\!\left(\frac{j_0-j}{2k}\right)
 \prod_{j=0}^{n-1}\Gamma\!\left(1-\frac{j+4}{2k}\right)
                    \Gamma\!\left(1-\frac{j+1}{2k}\right)}
 {\displaystyle
 \Gamma\!\left(1-\frac1{2k}\right)^n
 \prod_{j=0}^{n-1}\Gamma\!\left(\frac12+\frac{n-j}{2k}\right)}.
\end{equation}
The vanishing gamma argument has derivative $k/2$, which accounts for $2/k$. Together, \eqref{eq:BC-derivative} and \eqref{eq:BC-endpoint-residue} evaluate the nonzero chamber sum in gamma and trigonometric values.

\begin{corollary}[Finite exception list for a fixed shift]\label{cor:fixed-shift}
Fix an even integer $\ell\geq2$ and put $C_\ell=\ell(\ell+2)$. The exceptional ambient ranks in Theorem~\ref{thm:BC-complete-endpoint} are exactly
\begin{equation}\label{eq:fixed-shift-ranks}
 r=\ell+1+\frac{C_\ell}{e},\qquad e\mid C_\ell,\quad 1\leq e\leq\ell.
\end{equation}
At this rank, the sign of the continued nonterminal sum is
\begin{equation}\label{eq:fixed-shift-sign}
 \sgn\mathcal K_{r-2}(q)=(-1)^{\ell/2+e}.
\end{equation}
There are exactly $\tau(C_\ell)/2$ exceptional ranks, where $\tau$ is the positive-divisor counting function. The smallest is $2\ell+3$ and the largest is $\ell^2+3\ell+1$.
\end{corollary}
\begin{proof}
Set $d=n+1-\ell>0$. Then
\[
 \frac1q=d+2\ell+2+\frac{C_\ell}{d},\qquad
 j_0=d-1-\frac{C_\ell}{d}.
\]
The endpoint condition is equivalent to $d\mid C_\ell$ and $d\geq\ell+2$. Indeed, the lower inequality $j_0\geq0$ is $d(d-1)\geq\ell(\ell+2)$, whose least integral solution is $d=\ell+2$; the upper inequality $j_0\leq n-1$ is then automatic. Writing $e=C_\ell/d$ gives exactly $e\mid C_\ell$ and $e\leq\ell$, and $r=n+2=d+\ell+1$. Substitution in \eqref{eq:BC-exception-sign} gives \eqref{eq:fixed-shift-sign}.

Finally, $\ell^2<C_\ell<(\ell+1)^2$. Its divisors occur in distinct complementary pairs, one not exceeding $\ell$ and the other at least $\ell+2$; $\ell+1$ does not divide $C_\ell=(\ell+1)^2-1$. Hence exactly half the divisors occur in \eqref{eq:fixed-shift-ranks}. Taking $e=\ell$ and $e=1$ gives the smallest and largest ranks.
\end{proof}

\begin{corollary}[The quadratic exceptions]\label{prop:B7}
For $\ell=2$, the only exceptional ambient ranks are $r=n+2=7$ and $11$. The continued nonterminal chamber sums satisfy
\[
 \mathcal K_5(1/12)<0,\qquad \mathcal K_9(1/15)>0.
\]
At all other ranks $r\geq4$ this quadratic chamber sum vanishes.
\end{corollary}
\begin{proof}
For the endpoint index,
\[
 j_0=n-2-\frac8{n-1}.
\]
Thus $n-1$ must divide $8$. The divisors $1,2$ give negative indices, whereas $4,8$ give $(n,j_0)=(5,1),(9,6)$. Formula~\eqref{eq:BC-exception-sign} gives the two signs.
\end{proof}

\subsection{Poles of the reference period in type \texorpdfstring{$D$}{D}}
For type $D$, write $K_0^D$ for \eqref{eq:singleton-endpoint} with $\varepsilon=0$, and put $\mathcal K_n^D=K_0^DF_0$. Unlike $F_1$, the normalized factor $F_0$ does not vanish at the shifted values.

\begin{proposition}[Poles of the type-$D$ reference period]\label{prop:D-endpoint}
Let $n\geq2$, let $\ell=2m$ satisfy $2\leq\ell<n+1$, and put
\[
 q=\frac{n+1-\ell}{n^2+3n+1},\qquad k=1/q.
\]
The function $\mathcal K_n^D$ has a simple pole at $q$ precisely when
\begin{equation}\label{eq:D-endpoint-condition}
 k\in\Z,\qquad n+3\leq k\leq2n+2.
\end{equation}
Otherwise it is holomorphic and nonzero there. For a fixed even shift $\ell$, its exceptional ambient ranks are exactly
\begin{equation}\label{eq:D-fixed-shift}
 r=\ell+1+\frac{\ell(\ell+1)-1}{e},\qquad
 e\mid\ell(\ell+1)-1,\quad1\leq e\leq\ell-1.
\end{equation}
In particular, the only quadratic exception is $D_8$ at $1/11$.
\end{proposition}
\begin{proof}
The first numerator gamma arguments of $K_0^D$ are
\[
 a_j=-\frac12+\frac{2n+2-j}{2}q,\qquad0\leq j\leq n-1.
\]
They lie in $(-1/2,1/2)$. All remaining numerator and denominator gamma arguments are positive; the denominator arguments depending on $j$ are $1/2+(n-j-1)q/2$. Thus a pole occurs only when a unique $a_{j_0}$ is zero, equivalently $j_0=2n+2-k\in\{0,\ldots,n-1\}$. It is simple, with reciprocal derivative $2/k$ in that gamma factor. The factor $F_0(q)$ is finite and nonzero by \eqref{eq:D-nonzero}, so it neither removes nor raises this pole.

For the fixed-shift form, set $d=n+1-\ell$ and $C=\ell(\ell+1)-1$. Then
\[
 k=d+2\ell+1+C/d,\qquad j_0=d-1-C/d.
\]
Condition \eqref{eq:D-endpoint-condition} is equivalent to $d\mid C$ and $d\geq\ell+1$. Writing $e=C/d$ gives exactly \eqref{eq:D-fixed-shift}. For $\ell=2$, $C=5$ and the only permitted small divisor is $e=1$, yielding $r=8$.
\end{proof}

The same exceptional ranks can be read directly from coincident parabolic candidates. Let a rank-$R$ system have a connected normal complement of the same classical type and rank $L$. Its degree-zero outer candidate is
\[
 \frac{R-L}{R(R-1)-L(L-1)}=\frac1{R+L-1}
 \quad\text{in type $D$},
 \qquad
 \frac{R-L}{R^2-L^2}=\frac1{R+L}
 \quad\text{in types $B,C$}.
\]
Equating this value with the degree-$\ell$ singleton candidate in rank $L$ and clearing denominators gives
\begin{equation}\label{eq:singleton-reflection}
 (L-\ell-1)(R-\ell-1)=
 \begin{cases}
  \ell(\ell+1)-1,&D,\\
  \ell(\ell+2),&B,C.
 \end{cases}
\end{equation}
The small divisor $e$ in \eqref{eq:D-fixed-shift} or \eqref{eq:fixed-shift-ranks} gives $L=\ell+1+e$. Thus the divisor lists classify precisely the candidate coincidences. A coincidence alone does not prove a pole; the complete Laurent coefficient must still be evaluated.

\Needspace{13\baselineskip}
\subsection{The quadratic insertion}\label{sec:quadratic-insertion}
After centering the normal gap, the quadratic Taylor coefficient contributes the following factor to the Selberg density. Put $y_i=(1-x_i)^{-1}$ and $p_j=\sum_i y_i^j$:
\[
 H_\varepsilon=\varepsilon q/4+q p_2-q p_1/2,
\]
where $\varepsilon=1$ for types $B,C$ and $\varepsilon=0$ for type $D$.

\begin{proposition}[Singleton quadratic insertion]\label{prop:singleton-insertion}
For the density $W$ in \eqref{eq:mixed-cube}, evaluated at $s=q$, define
$\nabla_iP=\partial_iP+P\partial_i\log W$ and
\begin{equation}\label{eq:singleton-divergence}
 P_i=x_i(A+By_i+Cp_1),
\end{equation}
where
\[
 A=\frac{1-n+q\{n^2+3n+6+\varepsilon(n+4)\}}{20(2q+1)},
 \quad B=\frac{4q}{5(2q+1)},\quad C=\frac{q}{5(2q+1)}.
\]
Then
\[
 H_\varepsilon-\kappa_\varepsilon=\sum_i\nabla_iP_i,
 \qquad
 \kappa_\varepsilon=\frac{\varepsilon q}{4}
 -\frac n2\{1-(n+3+\varepsilon)q\}A.
\]
At the quadratic candidate,
\begin{align}
 \kappa_{BC}&=\frac{(n-1)(n+2)(3n+5)}{40(n+1)(n+3)(n^2+6n+1)}>0,\label{eq:kappa-BC}\\
 \kappa_D&=-\frac{n(n-1)(n+4)}{8(n^2+3n+1)(n^2+5n-1)}<0.\label{eq:kappa-D}
\end{align}
Consequently, the projective period of the insertion is $\kappa_\varepsilon$ times the corresponding uninserted chamber period.
\end{proposition}
\begin{proof}
Write
\[
 \alpha=\frac{1-\varepsilon q}{2},\qquad
 \beta=1-2q,\qquad \gamma=-\frac q2,
 \qquad h=\alpha+\beta-1+\gamma(n-1).
\]
Expanding $\sum_i\nabla_iP_i$ in the symmetric functions $p_1,p_2$ gives
\begin{align*}
 \sum_i\nabla_iP_i={}&
 \{C+B(2-\beta-\gamma)\}p_2
 +\{C(1-\beta)+\gamma B\}p_1^2\\
 &+\{B(\alpha+\beta-2)-(\beta-1)A+C(nh-1)\}p_1+nAh.
\end{align*}
Substitution of the displayed values of $A,B,C$ reduces the three coefficients to $q,0,-q/2$, respectively. Since $h=\{1-(n+3+\varepsilon)q\}/2$, the constant term is $\varepsilon q/4-\kappa_\varepsilon$. This proves the pointwise identity. After multiplication by the density, its right-hand side is an exact twisted differential. For generic parameters each loaded chamber admits the regularization described in Section~\ref{sec:integral-conventions}, so the period of this differential is zero. Meromorphic continuation of the complete identity gives the displayed period formula at the shifted parameters; no common domain of ordinary convergence for the unbounded chambers is required.
\end{proof}
Since $\kappa_{BC}\ne0$, cancellation in the $B/C$ calculation comes from the chamber sum, not from the insertion itself.

\section{Quadratic residues in types \texorpdfstring{$B$}{B}, \texorpdfstring{$C$}{C}, and \texorpdfstring{$D$}{D}}\label{sec:actual-singletons}
We now use the chamber evaluations to prove Theorems~\ref{thm:D-family} and~\ref{thm:BC-family}. We first list every support at the candidate value, then compute its coefficient with the original lattice measure.

\subsection{The contributing supports}

\begin{lemma}[Complete quadratic singleton support list]\label{lem:singleton-census}
Every support with candidate value $q_r^D$ has a singleton normal complement and Taylor degree two, except at $D_8$. There is then one additional degree-zero complement: the connected $D_4$ subdiagram containing the fork. This case has four flags of length two.

Every support with candidate value $q_r^{BC}$ likewise has a singleton complement and degree two, except in ranks seven and eleven. Rank seven also has the connected $B_5$ or $C_5$ complement containing the terminal node; rank eleven has the corresponding $B_4$ or $C_4$ complement. Their Taylor degree is zero. The two cases have five and four flags of length two, respectively. No flag at these candidates has length three.
\end{lemma}
\begin{proof}
Put $u=r-3$. In type $D$, the total root count is $N=r(r-1)$ and $M=N-1=u^2+5u+5$. In types $B,C$, $N=r^2$ and $M=N-1=u^2+6u+8$. Let a competing candidate have normal rank $j$, normal root count $N_J$, and Taylor order $a\geq0$. If $g=\gcd(u,M)$, the equality of candidates implies
\[
 N-N_J=tM/g,\qquad r-j-a=tu/g,
 \qquad1\leq t\leq g.
\]
The last bound follows from $N=M+1$ and $M/g>1$. For $t=g$, necessarily $N_J=1$, so $J$ is a singleton and $a=2$.

For $t<g$, write $x=(g-t)/g$. Then
\[
 N_J=1+xM,\qquad j+a=3+xu.
\]
In type $D$, $g\mid5$. Such a competing subsystem has $N_J\geq12$ and hence $j\geq4$. A parabolic subsystem of $D_r$ of rank $j\geq3$ has at most $j(j-1)$ positive roots: its components are of types $A$ and at most one $D$, and the bound follows by merging their rank counts. Thus
\[
 1+x(u^2+5u+5)\leq(3+xu)(2+xu),
 \qquad xu^2\leq5.
\]
Since $x\geq1/5$, $u\leq5$. The only case with $g>1$ is $u=5$, and only $x=1/5$ survives the inequality. It gives $r=8$, $j\leq4$, $N_J=12$, hence precisely $J=D_4$ and $a=0$.

For $B,C$, $g\mid8$ and $N_J\leq j^2$, so
\[
 1+x(u^2+6u+8)\leq(3+xu)^2,
 \qquad xu^2\leq8.
\]
Now $u\leq8$. The surviving divisibility tests are
\[
\begin{array}{c|c|c|c}
 r&x&N_J&j\text{ at most}\\\hline
 5&1/2&13&4\\
 7&1/4&13&4\\
 7&1/2&25&5\\
 11&1/8&16&4.
\end{array}
\]
The first two rows are impossible: a rank-four $B/C$ parabolic is either the connected terminal subdiagram, with sixteen positive roots, or has at most ten. The last two force the connected tail of rank five and four, respectively, and $a=0$.

A tail of the stated orthogonal type has only one placement in the Dynkin diagram. An outer support is contained in a singleton-complement support precisely when the omitted singleton lies in that tail. The chain counts are therefore four, five, and four. There are at most two support sizes, so no chain has length three.
\end{proof}

\subsection{The type-\texorpdfstring{$D$}{D} wall integrals}\label{sec:D-normalization}
We compute the quadratic period with its lattice normalization. In this subsection set $n=r-2$, $M=n^2+3n+1$, $q=q_r^D=(n-1)/M$, and
\[
 \alpha_0=-\frac12+(n+1)q,\qquad
 \kappa_D=-\frac{n(n-1)(n+4)}{8(n^2+3n+1)(n^2+5n-1)}.
\]
For $r\ne8$, the residue is
\begin{equation}\label{eq:D-singleton-residue}
 \boxed{\displaystyle
 R_r^D=\frac{\zeta(q-2)}{M}\,2^{2-n-q}\kappa_D
 \frac{S_n(\alpha_0,1-2q,-q/2)}{n!}\,F_0(q).}
\end{equation}
Here $F_0(q)$ is the sine product \eqref{eq:D-nonzero} with $m=1$.

To prove \eqref{eq:D-singleton-residue}, use the orthogonal coordinates
\[
 t_1>\cdots>t_{r-1}>|t_r|,\qquad
 P_D(t)=\prod_{i<j}(t_i^2-t_j^2).
\]
The map from $t$ to the simple-root gaps has absolute determinant two. The two fork walls are $t_{r-1}=t_r$ and $t_{r-1}=-t_r$; the other walls are $t_i=t_{i+1}$ for $1\leq i\leq r-2$. At a difference wall write $t_i=c+a/2$, $t_{i+1}=c-a/2$, and let $v_j$ be the remaining coordinates. Removing the normal factor $a$ gives the exterior polynomial
\begin{equation}\label{eq:D-wall-polynomial}
 Q(c,v)=2c\prod_{j=1}^{n}(c^2-v_j^2)^2
              \prod_{j<k}(v_j^2-v_k^2),
\end{equation}
with absolute values on each chamber. Its centered quadratic coefficient is
\[
 Q^{-s}\frac{s}{c^2}\sum_j
 \left(\frac1{(1-v_j^2/c^2)^2}-\frac1{2(1-v_j^2/c^2)}\right).
\]
The comparison with the actual normal coordinate is the single-parameter version of Lemma~\ref{lem:centering}. Its primitive is again \eqref{eq:projective-primitive}, with one normal gap and exterior dimension $r-1$. The remaining simple roots supply the adjacent-gap forms; the root whose simple-root coefficients are all one supplies their total sum. Thus the independent-exponent convergence argument applies on these walls too.

Fix $c=1$ and put $x_j=v_j^2$. Before inversion, the mixed-chamber parameters are $(1/2,1-2q,-q/2)$. The normalization at a nonfork wall is the product
\[
\begin{aligned}
 &2\;\text{(simple-root Jacobian)}
 \times 2\;\text{(two signs of the last coordinate)}\\
 &\qquad\times 2^{-n}\;\text{(squaring Jacobians)}
 \times 2^{-q}\;\text{(the factor $2c$ in $Q$)}
 =2^{2-n-q}.
\end{aligned}
\]
The nonfork walls give the mixed chambers $p=n,n-1,\ldots,1$. Each of the two fork walls contributes half of the same normalization at $p=0$; together they give exactly one $p=0$ chamber. Therefore the sum of the uninserted wall periods is
\[
 2^{2-n-q}K_0^D(q)F_0(q)
 =2^{2-n-q}\frac{S_n(\alpha_0,1-2q,-q/2)}{n!}F_0(q),
\]
where the equality is \eqref{eq:singleton-endpoint} after inversion of the reference chamber. Proposition~\ref{prop:singleton-insertion} replaces the insertion by $\kappa_D$. Summing over the original normal gap gives $\zeta(s-2)$, and the diagonal denominator has slope $M$. This proves \eqref{eq:D-singleton-residue} whenever Lemma~\ref{lem:singleton-census} leaves only the singleton sources.

\subsection{The difference walls in types \texorpdfstring{$B$}{B} and \texorpdfstring{$C$}{C}}\label{sec:BC-normalization}
In this subsection reset $n=r-2$ and $q=q_r^{BC}=(r-3)/(r^2-1)$. A nonterminal wall is a difference wall $t_i=t_{i+1}$, $1\leq i<r$; the terminal wall is $t_r=0$. In the convention \eqref{eq:def-xi}, positive orthogonal coordinates satisfy $t_1>\cdots>t_r>0$, and
\[
 P_C(t)=\prod_i t_i\prod_{i<j}(t_i^2-t_j^2),\qquad P_B(t)=2^rP_C(t).
\]
For $C$, the simple-coroot gaps are $m_i=t_i-t_{i+1}$ for $i<r$ and $m_r=t_r$, with Jacobian one. For $B$ they have $m_r=2t_r$, with Jacobian two. At a difference wall set $t_i=c+a/2$, $t_{i+1}=c-a/2$. Removing the normal factor $a$ and setting $a=0$ leaves
\[
 Q_C(c,v)=2c^3\prod_{j=1}^n v_j
          \prod_j(c^2-v_j^2)^2\prod_{j<k}(v_j^2-v_k^2),\qquad
 Q_B=2^rQ_C.
\]
The centered quadratic coefficient is $Q_C^{-s}c^{-2}H_1$, with $H_1$ from Section~\ref{sec:quadratic-insertion}. Fixing $c=1$ and squaring the remaining positive coordinates gives the parameters $((1-q)/2,1-2q,-q/2)$ and the common factor $2^{-n-q}$. Each of the $r-1=n+1$ difference walls contributes one of $K_n,K_{n-1},\ldots,K_0$, once. Therefore the complete nonterminal projective coefficients are
\begin{equation}\label{eq:BC-actual-sum}
 \cA_{C,\mathrm{nt}}=2^{-n-q}\kappa_{BC}\mathcal K_n(q),\qquad
 \cA_{B,\mathrm{nt}}=2^{1-rq}\cA_{C,\mathrm{nt}},
\end{equation}
where $\mathcal K_n(q)$ is the value of the continued product $K_0F_1$, as in Theorem~\ref{thm:BC-complete-endpoint}. At a singular endpoint it is not the product of two separately specialized values. The normal lattice factor $\zeta(q-2)$ and the radial factor $1/(r^2-1)$ are the same for all these nonterminal walls. These common factors show that the chamber cancellation has the required wall multiplicities.

\subsection{The terminal wall in types \texorpdfstring{$B$}{B} and \texorpdfstring{$C$}{C}}
After removal of its normal root factor, the exterior density, apart from a nonzero constant, is
$W=\prod_i t_i^{-3s}\prod_{i<j}(t_i^2-t_j^2)^{-s}$. Grouping the two contributions from each pair $i,j$ gives
\begin{equation}\label{eq:BC-terminal-quadratic}
 s\sum_i t_i^{-2}
 =-\frac{s}{1+3s}\sum_i
 \left(\partial_i(t_i^{-1})+t_i^{-1}\partial_i\log W\right).
\end{equation}
After multiplication by $W$, the right-hand side is an exact twisted differential. For generic parameters its period on the regularized loaded chamber is zero; meromorphic continuation of the entire identity therefore gives zero at the critical exponent.

\subsection{Completion of the pole and holomorphy proof}
\begin{proof}[Proof of Theorems~\ref{thm:D-family} and~\ref{thm:BC-family}]
First take $q=q_r^D$. For $r\ne8$, Proposition~\ref{prop:D-endpoint} shows that the Selberg endpoint in \eqref{eq:D-singleton-residue} is finite and nonzero. The scalar $\kappa_D$ and the sine product are nonzero, and $\zeta(q-2)<0$. Lemma~\ref{lem:singleton-census} excludes all other target supports and higher-order terms. Hence the point is a simple pole.

The only possible negative gamma factors of the endpoint have arguments
$-1/2+(2n+2-j)q/2$, $0\leq j\leq n-1$. Their sign changes at
\[
 j_*=n-2-\frac5{n-1}.
\]
There are two negative factors for $n=2$ and for $n\geq7$, and three for $n=3,4,5$. The omitted case $n=6$ is precisely the resonant case $D_8$. All other gamma factors are positive. The product $\zeta(q-2)\kappa_DF_0(q)$ is positive, proving the residue signs in the theorem.

At $D_8$, the additional normal complement is the $D_4$ subdiagram containing the fork. The outer support is $\{1,2,3,4\}$, with count $44$; the four singleton supports in that complement have root count $11$. Their cumulative multiplier is $11/(44\cdot55)=1/220$. On the outer face, the restricted polynomial is
\[
 \prod_{i=1}^4 y_i^8\prod_{i<j\leq4}(y_i^2-y_j^2),
 \qquad y_1>y_2>y_3>y_4>0.
\]
Fixing $y_1=1$ and squaring the other three variables gives the ordinary positive period
\[
 \frac1{48}S_3(3/22,10/11,-1/22).
\]
The product with $R_4^D>0$ and the cumulative multiplier proves \eqref{eq:D8-main}. All four target flags have been included and no longer chain exists. Thus the pole order is exactly two.

Now take $q=q_r^{BC}$. Outside ranks $7,11$, Corollary~\ref{prop:B7} gives $\mathcal K_n(q)=0$ with a regular endpoint. Equation~\eqref{eq:BC-actual-sum} and the terminal identity \eqref{eq:BC-terminal-quadratic} therefore make the complete quadratic coefficient zero. Lemma~\ref{lem:singleton-census} then proves holomorphy. At rank seven the only length-two terms come from the complementary $B_5/C_5$ coefficient at $1/12$; at rank eleven they come from the complementary $B_4/C_4$ coefficient at $1/15$.

The nonexceptional part of the argument has already proved that these lower-rank complementary residues vanish. Every potentially contributing support in the complement has the same root count ($24$ in rank five and $15$ in rank four). Lemma~\ref{lem:two-step}, including its zero-residue case, consequently makes the aggregate double coefficient zero. The remaining pole order is at most one; its simple coefficient is not determined by this argument.
\end{proof}

For an explicit base value in \eqref{eq:D8-main}, set $h=\pi/22$. Formula~\eqref{eq:D-singleton-residue} gives
\begin{equation}\label{eq:D4-base}
 R_4^D=\frac{\zeta(-21/11)}{11}\,2^{-1/11}
 \left(-\frac3{286}\right)
 \frac{S_2(-5/22,9/11,-1/22)}2
 \frac{\sin h\,\sin^2(2h)}{4\prod_{j=1}^4\cos(jh)}>0.
\end{equation}
Thus every factor in the double-pole coefficient \eqref{eq:D8-main} is explicitly evaluated.

\Needspace{12\baselineskip}
\section{Conclusion}
Theorem~\ref{thm:main} gives a shifted pole at a positive real point in every type-$A$ rank $r\geq5$, with double poles only at $A_{12}$ and $A_{20}$. Theorem~\ref{thm:D-family} gives the corresponding type-$D$ family, with its only double pole at $D_8$. The leading coefficients at all three resonant ranks are explicit. Theorem~\ref{thm:BC-family} proves that the analogous quadratic candidates in types $B$ and $C$ are removable outside ranks $7$ and $11$; at those ranks the possible double terms also cancel.

The remaining simple coefficients at $B_7,C_7,B_{11},C_{11}$ are not determined here. The higher-even-degree results classify continued chamber sums, not the full Witten coefficients: the normal insertion and every incident support must still be included. These are the natural remaining cases.

\appendix
\section{Verification scope}\label{app:verification}
The ancillary programs perform exact symbolic checks of the finite sine, divergence, and $q$-difference identities, together with root and support enumerations in the stated finite ranges. Separate high-precision calculations compare equivalent formulas for the three double-pole coefficients; these comparisons are diagnostic, not proofs of nonvanishing.

Run \path{anc/verify_main.py} with Python for the exact checks; pass \texttt{--diagnostics} to add the numerical comparisons. These programs verify finite algebra and enumerations. They do not prove the residue theorem quoted in Section~\ref{sec:local}, nor do they replace the meromorphic-continuation and Stokes arguments in the text.

\end{document}